\documentclass[a4paper,twoside,reqno]{amsart}

\usepackage[pagewise]{lineno}
\usepackage[utf8]{inputenc}

\usepackage{authblk}
\usepackage{hyperref}
\hypersetup{urlcolor=blue, citecolor=red}
\usepackage{amsmath,amsthm,amssymb,xcolor,bbm,mathrsfs}
\usepackage[english]{babel}
\usepackage{fancyhdr}
\newcommand{\norm}[1]{\|#1\|}
\newcommand{\n}{\hspace*{-6pt}}

\newcommand{\E}{\mathbf{E}}
\newcommand{\RomanNumeralCaps}[1]
    {\MakeUppercase{\romannumeral #1}}

\DeclareMathOperator{\diag}{diag}
\DeclareMathOperator{\unif}{uniform}
\DeclareMathOperator{\Span}{Span}

\makeatletter
\@namedef{subjclassname@1991}{2020 Mathematics Subject Classification}
\makeatother

\title{Mixed Hegselmann-Krause Dynamics \RomanNumeralCaps 2}
\author{Hsin-Lun Li}

\date{}
\email{hsinlunl@asu.edu}

\theoremstyle{definition}
\newtheorem{theorem}{Theorem}
\newtheorem{definition}[theorem]{Definition}
\newtheorem{lemma}[theorem]{Lemma}
\newtheorem{corollary}[theorem]{Corollary}

\begin{document}

\allowdisplaybreaks

\thispagestyle{firstpage}
\maketitle
\begin{center}
    Hsin-Lun Li
\end{center}

\begin{abstract}
The mixed Hegselmann-Krause (HK) model consists of a finite number of agents characterized by their opinion, a vector in $\mathbf{R^d}$. For the deterministic case, each agent updates its opinion by the rule: decide its degree of stubbornness and mix its opinion with the average opinion of its neighbors, the agents whose opinion differs by at most some confidence threshold from its opinion at each time step. The mixed model is studied deterministically in \cite{mHK}. In this paper, we study it nondeterministically and involve a social relationship among the agents which can vary over time. We investigate circumstances under which asymptotic stability holds. Furthermore, we indicate the mixed model covers not only the HK model but also the Deffuant model.
\end{abstract}

\section{Introduction}
The mixed Hegselmann-Krause (HK) model originated from the HK model. The original HK model comprises a finite set of agents characterized by their opinion, a number in $[0,1]$. Agent $i$ updates its opinion $x_i$ by taking the average opinion of its neighbors, the agents whose opinion differs by at most some confidence threshold from $x_i$. There are two types of HK models: the synchronous HK model and the asynchronous HK model. For the synchronous HK model, all agents update their opinion at each step, whereas for the asynchronous HK model, only one agent uniformly selected at random updates its opinion at each time step. For the mixed model, all agents can decide their degree of stubbornness and mix their opinion with the average opinion of their neighbors. In \cite{mHK}, the mixed model is studied deterministically as follows. Given a confidence threshold $\epsilon>0$ and a set $[n]=\{1,2,\ldots,n\}$ including all agents, $$x_i(t+1)=\alpha_i(t)x(t)+\frac{1-\alpha_i(t)}{|N_i(t)|}\sum_{k\in N_i(t)}x_k(t)$$ where
$$\begin{array}{rcl}
    x_i(t) &\n\in\n& \mathbf{R^d}\ \hbox{is the opinion of agent $i$ at time $t$}, \vspace{2pt} \\
    \alpha_i(t) &\n\in\n& [0,1]\ \hbox{is the degree of stubbornness of agent $i$ at time $t$}\ ,\vspace{2pt} \\
    N_i(t) &\n=\n& \{j\in [n]:\norm{x_i(t)-x_j(t)}\leq \epsilon\}\ \hbox{is the neighborhood of agent $i$ at time $t$.}
\end{array}$$ An agent is \emph{absolutely stubborn} if its degree of stubbornness equals 1. In this paper, we study it nondeterministically and involve a social relationship which can vary over time. Interpreting in graph, each vertex stands for an agent. The edge connecting two vertices indicates a relationship between the two, such as a friendship. We involve two types of graphs: the social graph and the opinion graph. The social graph depicts the social relationship in which two vertices are social neighbors if and only if they are socially connected, whereas the opinion graph depicts the opinion relationship in which two vertices are opinion neighbors if and only if they are at a distance of at most $\epsilon$ apart. In words for the nondeterministic mixed model, an agent updates its opinion by the mechanism: decide its degree of stubbornness and mix its opinion with the average opinion of its social and opinion neighbors. Before detail the mechanism of the nondeterministic mixed model and how it also works for the Deffuant model, we introduce some terms and the Deffuant model.

\begin{definition}{\rm
A \emph{social graph} at time $t$, $G(t)=([n],E(t))$, is an undirected graph with vertex set and edge set,
$$[n]\ \hbox{and}\ E(t)=\{(i,j)\in [n]^2: \hbox{vertices $i$ and $j$ are socially connected}\}.$$
A \emph{social graph for update} at time $t$, $\Tilde{G}(t)=([n],\Tilde{E}(t))$, is a subgraph of the social graph at time $t$ for the opinion update.
An \emph{opinion graph} at time $t$, $\mathscr{G}(t)=([n],\mathscr{E}(t))$, is an undirected graph with vertex set and edge set,
$$[n]\ \hbox{and}\ \mathscr{E}(t)=\{(i,j)\in [n]^2: i\neq j\ \hbox{and}\ \norm{x_i(t)-x_j(t)}\leq \epsilon\}.$$ A \emph{profile} at time $t$, $\tilde{G}(t)\cap \mathscr{G}(t)$, is the intersection of the social graph for update and the opinion graph at time $t$.}
\end{definition}

\begin{definition}{\rm
An opinion graph $\mathscr{G}$ is \emph{$\delta$-trivial} if any two vertices in $\mathscr{G}$ are at a distance of at most $\delta$ apart.}
\end{definition}

\begin{definition}{\rm 
A \emph{matching} $M$ in graph $G$ is a set of pairwise nonadjacent edges, none of which are loops.}
\end{definition}

\begin{definition}{\rm
 A symmetric matrix~$M$ is called a \emph{generalized Laplacian} of a graph $G = (V, E)$ if for~$x, y \in V$, the following two conditions hold:
 $$ M_{xy} = 0 \ \hbox{for} \  x \neq y \ \hbox{and} \ (x,y) \notin E \quad \hbox{and} \quad M_{xy} < 0 \ \hbox{for} \ x \neq y \ \hbox{and} \ (x,y) \in E. $$
 Let~$d_G (x)$ = degree of~$x$ in $G$, let~$V (G)$ = vertex set of~$G$, and let~$E (G)$ = edge set of~$G$.
 Then, the \emph{Laplacian} of~$G$ is defined as~$\mathscr{L} = D_G - A_G$ where
 $$ D_G = \diag ((d_G (x))_{x \in V (G)}) \quad \hbox{and} \quad A_G = \ \hbox{the adjacency matrix}. $$
 In particular, $(A_G)_{xy} = \mathbbm{1} \{(x,y) \in E (G)\}$ when the graph~$G$ is simple.}
\end{definition}

The original Deffuant model consists of a finite number of agents characterized by their opinion, a number in $[0,1]$. Two socially connected agents are uniformly selected at random and approach each other at a rate $\mu\in [0,1/2]$ if and only if their opinion distance does not exceed some confidence threshold, say $\epsilon>0$. The model is as follows:
$$\begin{array}{rcl}
    x_i(t+1) &\n=\n& x_i(t)+\mu(x_j(t)-x_i(t))\mathbbm{1}\{\norm{x_i(t)-x_j(t)}\leq \epsilon\},\vspace{2pt}  \\
      x_j(t+1) &\n=\n& x_j(t)+\mu(x_i(t)-x_j(t))\mathbbm{1}\{\norm{x_i(t)-x_j(t)}\leq \epsilon\}.
\end{array}$$
Now, we consider two interaction mechanisms: the pair interaction and the group interaction. For the pair interaction, a selected pair of agents can approach each other at distinct rates if and only if their opinion distance does not exceed $\epsilon$, whereas for the group interaction, an agent can decides its degree of stubbornness and mix its opinion with the average opinion of its social and opinion neighbors.
The difference between the Deffuant model and the HK model is their interaction mechanism. The former is pair interaction, whereas the latter is group interaction.  

For nondeterministic mixed model, we assume that 
\begin{itemize}
    \item $\alpha_i(t)$, $t\geq 0$ are independent and identically distributed random variables on $[0,1]$ for all $i\in [n]$, \vspace{2pt}
    \item $U_t$, $t\geq 0$ are independent and identically distributed random variables with a support $S$. For the pair interaction, $U_t=\{(i,j)\in [n]^2: \alpha_i(t)<1\ \hbox{or}\ \alpha_j(t)<1\}$ and $S\subset \{\hbox{all matchings in}\ G(0)\}$, whereas for the group interaction, $U_t=\{i\in [n]:\alpha_i(t)<1\}$, $S\subset \mathscr{P}([n])$, the power set of $[n]$, and $V(S)=\{i: i\in a\ \hbox{for some}\ a\in S\}\supset [n]$, \vspace{2pt}
    \item $\Tilde{E}(t)=U_t\cap E(t)$ for the pair interaction, whereas $\tilde{E}(t)=E(t)$ for the group interaction, \vspace{2pt}
    \item $(\Omega,\mathscr{F},P)$ is a probability space for $\mathscr{F}\subset\mathscr{P}(\Omega)$ a $\sigma$-algebra and $P$ a probability measure. 
\end{itemize}

The update rule goes as follows:
\begin{equation*}
 x(t+1)=\alpha_i(t)x_i(t)+\frac{1-\alpha_i(t)}{|\mathscr{N}_i(t)|}\sum_{k\in \mathscr{N}_i(t)}x_k(t)   
\end{equation*}
 where 
$\mathscr{N}_i(t)=\{j\in [n]: (i,j)\in \tilde{E}(t)\cap \mathscr{E}(t)\}$ is the collection of social and opinion neighbors of agent $i$ for opinion update. Express in matrix,
\begin{equation}
\label{mHK-nondeter}
 x (t + 1) = \diag (\alpha (t)) \,x (t) + (I - \diag (\alpha (t))) \,A (t) \,x (t)
\end{equation}
 where $A (t) \in \mathbf{R^{n\times n}}$ is row stochastic with
 $$ A_{ij} = \mathbbm{1} \{j \in \mathscr{N}_i (t) \} / |\mathscr{N}_i(t)| $$
 and where
 $$ \begin{array}{rclcl}
      x (t) & \n = \n & (x_1 (t), x_2 (t), \ldots, x_n (t))' & \n = \n & \hbox{transpose of} \ (x_1 (t), x_2 (t), \ldots, x_n (t)), \vspace*{4pt} \\
      \alpha (t) & \n = \n & (\alpha_1 (t), \alpha_2 (t), \ldots, \alpha_n (t))' & \n = \n & \hbox{transpose of} \ (\alpha_1 (t), \alpha_2 (t), \ldots, \alpha_n (t)). \end{array} $$ In particular, \eqref{mHK-nondeter} reduces to
\begin{itemize}
    \item the synchronous HK model if $G(t)$ is complete, $S=\{[n]\}$ and $\alpha_i(t)=0$ for all $i\in [n]$ and $t\geq 0$, \vspace{2pt} 
    \item the asynchronous HK model if $G(t)$ is complete, $S=\big\{\{i\}\big\}_{i\in [n]}$, $U_t$ is a uniform random variable on $S$, denoted by $U_t=\unif{(S)}$, and $\alpha_i(t)=0$ if $i\in U_t$ at all times, and \vspace{2pt}
    \item the Deffuant model if $E(t)=E$, $S=\big\{\{(i,j)\}\big\}_{(i,j)\in E}$, $U_t=\unif{(S)}$ and $\alpha_i(t)=\alpha_j(t)=1-2\mu$ for all $(i,j)\in U_t$ and $t\geq 0$.
    
\end{itemize}


\section{Main results}
The power of the mixed model is that it can address at least two interaction mechanisms: pair interaction and group interaction. For the Deffuant model, the convergence parameter $\mu$ is constant, namely that the two selected socially connected agents can only approach each other at the same rate. However, for the pair interaction in the mixed model, all agents can decide their convergence parameter at all times. We involve the term ``matching'' in graph theory to express not only a pair of agents can update their opinion. Apart from the deterministic case in  \cite{mHK} that all agents are considered, only those agents who are not absolutely stubborn are considered in the nondeterministic case, which stands to reason since the update occurs only on the non-absolutely stubborn.

\begin{theorem}\label{Thm:fin}
Assume that $0\leq\limsup_{t\to\infty}\sup\{\alpha_i(t):i\in [n]\ \hbox{and}\ \alpha_i(t)<1\}<1$ almost surely. Then, all components of a profile are $\delta$-trivial in finite time almost surely for all~$\delta>0$, i.e.,
 $$ \tau_{\delta} := \inf \{t \geq 0 : \hbox{all components of $\tilde{G}(t)\cap \mathscr{G} (t)$ are $\delta$-trivial} \} < \infty\ \hbox{almost surely}.$$
\end{theorem}

Corollary~\ref{cor:asym} depicts circumstances under which asymptotic stability holds. It turns out that agents in the same component for infinitely many times achieve their consensus.
\begin{corollary}\label{cor:asym}
 Assume that~$0\leq \sup_{t\in \mathbf{N}} \sup \{\alpha_i (t):i\in [n]\ \hbox{and}\ \alpha_i(t)<1\}< 1$ almost surely.
Then, all components of a profile are $\delta$-trivial after some finite time, i.e., asymptotic stability holds in \eqref{mHK-nondeter}. 
\end{corollary}
Next, we show circumstances under which a consensus can be achieved. The author in \cite{mHK} has indicated that an opinion graph preserves $\delta$-triviality for all $\delta>0$, therefore when $\mathscr{G}(t)$ is $\epsilon$-trivial and the social graph for update is connected for infinitely many times, the profile is connected for infinitely many times. Hence, a consensus can be achieved given that asymptotic stability holds in this case, which elaborates Corollary \ref{cor:consensus}.
\begin{corollary}\label{cor:consensus}
 Assume that~$0\leq \sup_{t\in \mathbf{N}} \sup \{\alpha_i (t):i\in [n]\ \hbox{and}\ \alpha_i(t)<1\}< 1$, the social graph for update is connected for infinitely many times and $\mathscr{G}(t)$ is $\epsilon$-trivial almost surely. Then, $$\lim_{t\in \infty}\max_{i,j\in [n]}\norm{x_i(t)-x_j(t)}=0$$
\end{corollary}

\section{The mixed model}
The proof of Theorem \ref{Thm:fin} goes through several steps. The general ideas are to find a monotone bounded function and construct an inequality involving the current opinion and the updated opinion. We introduce several lemmas before the proof.
\begin{lemma}[\cite{mHK}]\label{NL8}
 Let~$Z(t)=\sum_{i,j\in [n]}\|x_i(t)-x_j(t)\|^2\wedge\epsilon^2.$ Then,~$Z$ is nonincreasing with respect to~$t$. In particular,
$$\begin{array}{rcl}
    \displaystyle Z(t)-Z(t+1)&\n \geq\n &\displaystyle 4 \sum_{i\in [n]}\bigg(1+|\mathscr{N}_i(t)|\frac{\alpha_i(t)}{1-\alpha_i(t)}\mathbbm{1}\{\alpha_i(t)<1\}\bigg)\|x_i(t)-x_i(t+1)\|^2.
\end{array}$$
\end{lemma}

\begin{lemma}[Perron-Frobenius for Laplacians \cite{TJP}]
\label{L5}
 Assume that~$M$ is a generalized Laplacian of a connected graph.
 Then, the smallest eigenvalue of~$M$ is simple and the corresponding eigenvector can be chosen with all entries positive.
\end{lemma}
\begin{lemma}[Courant-Fischer Formula \cite{RC}]
\label{L6}
 Assume that~$Q$ is a symmetric matrix with eigenvalues~$\lambda_1 \leq \lambda_2 \leq \cdots \leq \lambda_n$ and corresponding eigenvectors~$v_1, v_2, \ldots, v_n$.
 Let~$S_k$ be the vector space generated by~$v_1, v_2, \ldots, v_k$ and~$S_0 = \{0 \}$. Then,
 $$ \lambda_k = \min \{x' Q x : \|x \| = 1, x \in S_{k - 1}^{\perp} \}. $$
\end{lemma}
\begin{lemma}[Cheeger's Inequality \cite{LPR}]
\label{L7}
 Assume that~$G = (V, E)$ is an undirected graph with the Laplacian~$\mathscr{L}$. Define
 $$ i (G) = \min \bigg\{\frac{|\partial S|}{|S|} : S \subset V, 0 < |S| \leq \frac{|G|}{2} \bigg\} $$
 where~$\partial S = \{(u,v) \in E : u \in S, v \in S^c\}$. Then,
 $$ 2i (G) \geq \lambda_2 (\mathscr{L}) \geq \frac{i^2(G)}{2 \Delta (G)} \quad \hbox{where} \quad \Delta (G) = \ \hbox{maximum degree of~$G$}. $$
\end{lemma}

\begin{lemma}[\cite{mHK}]
\label{L10}
 Assume that~$Q$ is a real square matrix and that~$V$ is invertible such that the matrix~$VQ = \mathscr{L}$ is the Laplacian of some connected graph.
 Then, 0 is a simple eigenvalue of $Q'Q$ corresponding to the eigenvector~$\mathbbm{1} = (1, 1,\ldots, 1)'$.
 In particular, we have
 $$ \lambda_2 (Q'Q) = \min \{x'Q'Qx : \|x \| = 1 \ \hbox{and} \ x \perp \mathbbm{1} \}. $$
\end{lemma}

\begin{lemma}\label{delta nontrivial lb}
 If some component $G$ in $\tilde{G}(t)\cap\mathscr{G}(t)$ is $\delta$-nontrivial and $\alpha_i(t)<1$ for all $i\in V(G)$, then
$$ \sum_{i \in V(G)} \,\|x_i (t) - x_i (t + 1) \|^2 > \frac{2 \delta^2(1 - \max_{i \in V(G)} \alpha_i (t))^2}{|V(G)|^8}. $$

\end{lemma}

\begin{proof}
Without loss of generality, we may assume that~$\tilde{G}(t)\cap \mathscr{G} (t)$ is connected.
 For~$\mathbbm{1} \in \mathbf{R^n}$ and~$W = \Span(\{\mathbbm{1} \})$, $\mathbf{R^n}=W\oplus W^\perp$.
 Then, write
 $$ x (t) = \left[c_1 \mathbbm{1} \,| \,c_2 \mathbbm{1} \,| \,\cdots \,| \,c_d \mathbbm{1} \right] +
            \left[\hat{c}_1 u^{(1)} \,| \,\hat{c}_2 u^{(2)} \,| \,\cdots \,| \,\hat{c}_d u^{(d)} \right] $$

 where $c_i$ and $\hat{c}_i$ are constants and $u^{(i)} \in \mathbbm{1}^\perp$ is a unit vector for all~$i \in [d]$.
 $$ \hbox{Claim:} \quad \sum_{k = 1}^d \,\hat{c}_k^2 > \frac{\delta^2}{2}. $$
 Assume by contradiction that this is not the case.
 Then, for all~$i, j \in [n]$,
 $$ \begin{array}{l}
    \displaystyle \|x_i (t) - x_j (t)\|^2 =
    \displaystyle \sum_{k = 1}^d \,\hat{c}_k^2 (u^{(k)}_i - u^{(k)}_j)^2 \vspace*{-4pt} \\ \hspace*{40pt} \leq
    \displaystyle \sum_{k = 1}^d \,\hat{c}_k^2 \ 2 ((u^{(k)}_i)^2 + (u^{(k)}_j)^2) \leq 2 \,\sum_{k = 1}^d \,\hat{c}_k^2 \leq \delta^2, \end{array} $$
contradicting the~$\delta$-nontriviality of~$\tilde{G}(t)\cap \mathscr{G}(t)$.
 Let~$B (t) = \diag (\alpha(t)) + (I - \diag (\alpha(t))) A (t)$. Then,
 $$ x (t) - x (t + 1) = (I - B (t)) \,x(t) = \left[\hat{c}_1 (I - B (t)) u^{(1)} \,| \,\cdots \,| \,\hat{c}_d (I - B(t)) u^{(d)} \right],$$
 from which it follows that
 $$ \sum_{i = 1}^n \,\|x_i (t) - x_i (t + 1) \|^2 = \sum_{j = 1}^d \,\hat{c}_j^2 \|(I - B (t)) u^{(j)} \|^2. $$
 Now, observe that
 $$ I - B (t) = (I - \diag (\alpha (t)))(I + D (t))^{-1} \mathscr{L} $$
 where~$\mathscr{L}$ is the Laplacian of~$\tilde{G}(t)\cap \mathscr{G} (t)$ and~$D (t)$ is diagonal with~$D_{ii} (t) = d_i (t)$, the degree of vertex~$i$.
 Assume that~$\alpha_i (t) < 1$ for all~$i \in [n]$.
 Then, $I - \diag (\alpha (t))$ is invertible, and according to~Lemmas~\ref{L7} and~\ref{L10},
 $$ \begin{array}{l}
    \displaystyle \|(I - B(t)) u^{(j)} \|^2 =
    \displaystyle u^{(j)'} (I-B(t))' (I-B(t))u^{(j)} \geq
    \displaystyle \lambda_2 ((I-B(t))' (I-B(t))) \vspace*{8pt} \\ \hspace*{25pt} =
    \displaystyle \lambda_2 \bigg(\mathscr{L} \,\diag\bigg(\bigg(\bigg(\frac{1-\alpha_i(t)}{1 + d_i(t)}\bigg)^2 \bigg)_{i = 1}^n \bigg) \mathscr{L} \bigg) \vspace*{8pt} \\ \hspace*{25pt} \geq
    \displaystyle \bigg(\frac{1 - \max_{i \in [n]} \alpha_i (t)}{n} \bigg)^2 \lambda_2 (\mathscr{L}^2) =
    \displaystyle \bigg(\frac{1 - \max_{i \in [n]} \alpha_i (t)}{n} \bigg)^2 \lambda_2^2 (\mathscr{L}) \vspace*{8pt} \\ \hspace*{25pt} >
    \displaystyle  \frac{4 (1 - \max_{i \in [n]} \alpha_i (t))^2}{n^8} \end{array} $$
 where we used that
 $$ \lambda_2 (\mathscr{L}) \geq \frac{i^2 (\mathscr{G}(t))}{2 \Delta(\mathscr{G} (t))} > \frac{(2/n)^2}{2n} = \frac{2}{n^3}. $$
In particular, we obtain
 $$ \sum_{i = 1}^n \,\|x_i (t) - x_i (t + 1) \|^2 > \frac{2 \delta^2(1 - \max_{i \in [n]} \alpha_i (t))^2}{n^8}. $$
\end{proof}

\begin{proof}[\bf Proof of Theorem~\ref{Thm:fin}]
By the assumption, there is $(t_k)_{k\geq 0}\subset\mathbf{N}$ strictly increasing such that $\max\{\alpha_i(t_k):i\in[n]\ \hbox{and}\ \alpha_i(t_k)<1\}\leq\gamma<1$ for some random variable $\gamma$ and for all $k\geq 0$. For all $m\geq 1$,
\begin{equation*}\label{tel}
    n^2\epsilon^2>Z(0)\geq Z(0)-Z(m)=\sum_{t=0}^{m-1}[Z(t)-Z(t+1)].
\end{equation*}
Letting $m\to\infty$ and applying Lemma \ref{NL8},
\begin{equation}\label{inq}
    n^2\epsilon^2\geq\sum_{t\geq 0}[Z(t)-Z(t+1)]\geq 4\sum_{t\geq 0}\sum_{i\in[n],\alpha_i(t)<1}\norm{x_i(t)-x_i(t+1)}^2.
\end{equation}
Set $\tau=\tau_{\delta}$. Assume by contradiction that $\tau=\infty$ on some $E\in \mathscr{F}$ with $P(E)>0$. Now, we separate to two parts: pair interaction and group interaction.\\

\noindent{\rm (i)}\ For group interaction, taking expectation on both sides on $E$, denoted by $\E_E$,
\begin{align}
    n^2\epsilon^2\ &\geq\ 4\sum_{t\geq 0}\E_E\bigg(\sum_{i\in [n],\alpha_i(t)<1}\norm{x_i(t)-x_i(t+1)}^2\bigg) \notag\\
    &=\ 4\sum_{t\geq0}\sum_{a\in S} \E_E\bigg(\sum_{i\in[n],\alpha_i(t)<1}\norm{x_i(t)-x_i(t+1)}^2\bigg|U_t=a\bigg)P(U_t=a) \notag\\
    &\geq\ 4\min_{a\in S}P(U_0=a)\sum_{t\geq 0}\sum_{a\in S}\E_E\bigg(\sum_{i\in [n],\alpha_i(t)<1}\norm{x_i(t)-x_i(t+1)}^2 \bigg| U_t=a \bigg) \notag\\
    &\geq\ 4\min_{a\in S}P(U_0=a)\sum_{t\geq 0}\E_E\bigg(\sum_{i\in [n]}\norm{x_i(t)-x_i(t+1)}^2 \bigg) \label{deter}\\
   &>\ 4\min_{a \in S}P(U_0=a)\sum_{t\geq 0}\E_E\bigg(\frac{2\delta^2(1-\max_{i\in [n]}\alpha_i(t))^2}{n^8}\bigg) \label{lb}\\
    &\geq\ 4\min_{a\in S}P(U_0=a)\sum_{k\geq 0} \E_E\bigg(\frac{2\delta^2(1-\max_{i\in [n]}\alpha_i(t_k))^2}{n^8}\bigg) \notag\\
    &\geq\ 4\min_{a\in S}P(U_0=a)\sum_{k\geq 0} \E_E\bigg(\frac{2\delta^2(1-\gamma)^2}{n^8}\bigg)=\infty,\ \hbox{a contradiction,} \notag
    \end{align}
    where $\min_{a\in S}P(U_0=a)>0$ is due to finite support $S$, \eqref{deter} follows $V(S)\supset [n]$, $\alpha_i(t)<1$ for all $i\in [n]$ and $t\geq 0$ during \eqref{deter}, and \eqref{lb} follows Lemma \ref{delta nontrivial lb}.\\
 
 \noindent{\rm (ii)}\ For pair interaction, since $\tau=\infty$ on $E$, there are $(i,j)\in \tilde{E}(t)\cap \mathscr{E}(t)$, $\alpha_i(t)<1$ and $\norm{x_i(t)-x_j(t)}>\delta$ for all $t\geq 0$, therefore from \eqref{inq},
 \begin{equation*}
     n^2\epsilon^2 >  4\sum_{k\geq 0}(1-\gamma)\delta/2=\infty,\ \hbox{a contradiction.}
 \end{equation*}
Therefore, time~$\tau$ is almost surely finite. 
\end{proof}

\begin{proof}[\bf Proof of Corollary~\ref{cor:asym}]
Assume by contradiction that a profile is $\delta$-nontrivial for infinitely many times on $E\in \mathscr{F}$ with $P(E)>0$. Then, there is a $\delta$-nontrivial component $G$ in $\tilde{G}(t_k)\cap \mathscr{G}(t_k)$ for $(t_k)_{\geq 0}\subset \mathbf{N}$ increasing. Let $\gamma=\sup_{t\in \mathbf{N}}\sup \{\alpha_i (t):i\in [n]\ \hbox{and}\ \alpha_i(t)<1\}$\\

\noindent{\rm (i)}\ For group interaction, following \eqref{deter} in the proof of Theorem \ref{Thm:fin},
\begin{align*}
   \eqref{deter}  &>\ 4\min_{a \in S}P(U_0=a)\sum_{k\geq 0}\E_E\bigg(\frac{2\delta^2(1-\max_{i\in V(G)}\alpha_i(t))^2}{|V(G)|^8}\bigg) \\
    &\geq\ 4\min_{a\in S}P(U_0=a)\sum_{k\geq 0} \E_E\bigg(\frac{2\delta^2(1-\max_{i\in [n]}\alpha_i(t_k))^2}{n^8}\bigg) \\
    &\geq\ 4\min_{a\in S}P(U_0=a)\sum_{k\geq 0} \E_E\bigg(\frac{2\delta^2(1-\gamma)^2}{n^8}\bigg)=\infty,\ \hbox{a contradiction,} \notag
\end{align*}

\noindent{\rm (ii)}\ For pair interaction, $G$ is a complete graph of order 2, say $V(G)=\{i,j\}$, therefore $\norm{x_i(t_k)-x_j(t_k)}>\delta$ and $\alpha_\ell(t_k)\leq \gamma$ for all $k\geq 0$ and for some $\ell\in V(G)$. Via \eqref{inq} in the proof of theorem \ref{Thm:fin}, 
\begin{equation*}
     n^2\epsilon^2 >  4\sum_{k\geq 0}(1-\gamma)\delta/2=\infty,\ \hbox{a contradiction.}
\end{equation*}
Therefore, all components of a profile are $\delta$-trivial after some finite time.
\end{proof}

\end{document}